\documentclass[12pt]{article}
\usepackage[utf8]{inputenc}
\usepackage[margin=1in]{geometry} 
\usepackage{amsmath,amsthm,amssymb, graphicx, multicol, array,xcolor,mathtools,mathrsfs,hyperref}
\usepackage[english]{babel}
\usepackage{csquotes}
\usepackage{comment}

\newtheorem{lemma}{Lemma}

\newtheorem{observation}{Observation}
\newtheorem{theorem}{Theorem}
\newtheorem{corollary}{Corollary}
\newtheorem{conjecture}{Conjecture}
\newtheorem{notation}{Notation}

\theoremstyle{definition}
\newtheorem{definition}{Definition}

\begin{document}
 
\title{Note on the union-closed sets conjecture and Reimer's average set size theorem}
\author{Kengbo Lu\thanks{kengbo.lu@cooper.edu} , Abigail Raz\thanks{abigail.raz@cooper.edu}}
\maketitle
\begin{abstract}
The Union-Closed Sets Conjecture, often attributed to Péter Frankl in 1979, remains an open problem in discrete mathematics. 
It posits that for any finite family of sets $\mathcal{S}\neq\{\emptyset\}$, if the union of any two sets in the family is also in the family, then \textit{there must exist an element that belongs to at least half of the member sets}. We will refer to the italicized portion as the abundance condition. In 2001, David Reimer proved that the average set size of a union-closed family $\mathcal{S}$ must be at least $\frac{1}{2}\log_{2}|\mathcal{S}|$. When proving this result, he showed that a family being union-closed implies that the family satisfies certain conditions, which we will refer to as the \textit{Reimer's conditions}. Therefore, as seen in the context of Tim Gowers' polymath project on the Union-Closed Sets Conjecture, it is natural to ask if all families that satisfy Reimer's conditions meet the abundance condition. A minimal counterexample to this question was offered by Raz in 2017. In this paper, we will discuss a general method to construct infinitely many such counterexamples with any fixed lower bound on the size of the member sets. Furthermore, we will discuss some properties related to these counterexamples, especially those focusing on how far these counterexamples are from being union-closed.
\end{abstract}
\section{Introduction}
The Union-Closed Sets Conjecture, stated below, often attributed to Péter Frankl in 1979, remains an open problem in extremal combinatorics.
\begin{conjecture}
    For any finite family of sets $\mathcal{S}\neq\{\emptyset\}$, if the union of any two sets in the family is also in the family, then there must exist an element that belongs to at least half of the sets. 
\end{conjecture} There are also multiple equivalent graph theoretic formulations and a lattice formulation for this conjecture \cite{Bruhn_Charbit_Telle_2013} \cite{Rival_1985}. 
An important early partial result was provided by Knill in 1994, who proved that in any non-trivial union-closed family with $n$ member sets, there exists an element that appears in at least $\frac{n-1}{\log_{2} n}$ of the sets \cite{knill1994graph}. Since then, many other partial results have been proven. For example, the conjecture holds for families with singletons or doubletons \cite{sarvate}, and it also holds for families whose largest member set size is no more than 11 \cite{Bošnjak_Marković_2008}. 
In 2022, Gilmer provided the first constant lower bound - there exists an element that appears in at least $0.01 n$ sets in every union-closed family \cite{gilmer2022constant}. Shortly after, multiple groups improve the constant lower bound to around 0.38 \cite{pebody2022extension,alweiss2022improved,sawin2023improved,yulei}. The best constant lower bound currently is approximately 0.38271, proven by Liu \cite{liu2023improving} by building upon the arguments of Yu \cite{yulei}, Sawin \cite{sawin2023improved}, and Gilmer\cite{gilmer2022constant}.%

Gilmer's groundbreaking constant lower bound is in some sense derived with an averaging argument together with the use of entropy. Prior to this, many partial results were derived from using averaging arguments, most notably including Reimer's result that the average set size of a union-closed family $\mathcal{S}$ must be at least $\frac{1}{2}\log_{2}|\mathcal{S}|$ \cite{reimer_2003}. To do this, Reimer introduced a set of necessary conditions, given in Definition \ref{Reimerdef}, and showed that any family satisfying these conditions has average set size at least $\frac{1}{2}\log_{2}|\mathcal{S}|$. In this paper we further comment on families satisfying Definition \ref{Reimerdef}.
Below we introduce some necessary notation and definitions.
\begin{notation} 
\textbf{ }
\begin{enumerate}
    \item $[n]:=\{1,\ldots,n\}$.
    \item Given $A\subseteq B \subseteq [n]$, $[A,B]\coloneq\{C:A\subseteq C \subseteq B\}$.
    
    \item Let $cl(\mathcal{S})$ denote the union closure of the family $\mathcal{S}$;  that is, $A\in cl(\mathcal{S})$ if and only if there exists some subset $\mathcal{T}\subseteq \mathcal{S}$ such that $A = \bigcup_{B \in \mathcal{T}}B$. 
\end{enumerate}
\end{notation}
\begin{definition}
\textbf{ }
    \begin{enumerate}
        \item A family of sets satisfies the \textit{abundance condition} if there exists an element that belongs to at least half of the sets. 
        \item Let $\mathcal{P}([n])$ be the power set of $[n]$; that is, $\mathcal{P}([n])$ is the set of all subsets of $[n]$, then $\mathcal{F}\subseteq\mathcal{P}([n])$ 
    is a \textit{filter} if $i\in[n]$ and $F\in\mathcal{F}$ implies $F\cup \{i\}\in\mathcal{F}$.
    \end{enumerate}
\end{definition}
\begin{definition}\label{Reimerdef}
$\mathcal{S}\subseteq\mathcal{P}([n])$ satisfies \textit{Reimer's conditions} if there exists a filter $\mathcal{F}\subseteq\mathcal{P}([n])$ and a bijection from $\mathcal{S}$ to $\mathcal{F}$, $A\rightarrow F_{A}$ such that:
\begin{enumerate}
    \item $A\subseteq F_{A}$ (subset condition);
    \item if $A,B\in \mathcal{S}$ and $ A\neq B$ then  $[A,F_{A}]\cap[B,F_{B}]=\emptyset$ (non-interference condition). 
\end{enumerate}
\end{definition}
\begin{lemma}
\label{reimerlemma}
\textup{(Lemma 1.3 in \cite{reimer_2003}) }If $\mathcal{S}$ is a non-trivial union-closed family of sets, then it satisfies Reimer's conditions.
\end{lemma}
As a part of Tim Gowers' polymath project on the union-closed sets conjecture in 2016, the following conjecture, a strengthening of the union-closed sets conjecture, was posed \cite{Gowers}.
\begin{conjecture}
\label{polymathcon}
    Any family that satisfies Reimer's conditions satisfies the abundance condition.
\end{conjecture}
Raz disproved this conjecture by constructing one such counterexample; that is, a family $\mathcal{S}$ that satisfies Reimer's conditions but not the abundance condition \cite{raz}. In her counterexample, $n=8$ and $|\mathcal{S}|=11$. She also proved that there is no counterexample with $n< 8$ or $|\mathcal{S}|<11$. 

\section{Observations}
This section includes observations on some necessary properties for a family to be a counterexample of Conjecture \ref{polymathcon}. Observations 1, 2, 4, 5 are proven in \cite{raz}.
\begin{observation}
    \normalfont{(Note 6 in \cite{raz}) }We may assume $\mathcal{F}$ must contain all sets in $\mathcal{P}([n])$ of size at least $n-1$.
\end{observation}

\begin{observation}
    \normalfont{(Fact 5 in \cite{raz}) }If  $\mathcal{S}$ satisfies the Reimer's conditions and every member of $\mathcal{F}$ has size at least $n-1$, then there exists an element satisfying the abundance condition.
\end{observation}

\begin{observation}
\label{unimapobs}
    The set in $\mathcal{S}$ that is mapped to $[n] \in \mathcal{F}$ must be $[n]$ itself.
\end{observation}
\begin{proof}
    Assume a set $A\in\mathcal{P}([n])$, $A\neq[n]$ is mapped to $[n]$. There must be an element $k\in[n]$ such that $k\notin A$. Let $F_{B}$ be the set in $\mathcal{F}$ of size $n-1$ such that $k\notin F_{B}$. Then $F_{B}\in[A,[n]]\cap[B,F_{B}]$, contradicting the non-interference condition.
\end{proof}
\begin{observation}
\label{extra sets observation}
    If $n$ is even, then we need at least 2 sets of size less than $n-1$
    in $\mathcal{F}$ to possibly construct a counterexample to Conjecture \ref{polymathcon}. If $n$ is odd, we need at least 3 sets of size less than $n-1$. 
\end{observation}
\noindent In the following discussion, we will focus on the smaller case in which $n$ is even. By the non-interference condition and Observation \ref{extra sets observation}, there must be at least two sets of size $n-2$ in $\mathcal{F}$. We further restrict to only consider families where the two necessary sets of size $n-2$ in $\mathcal{F}$ are the only sets in $\mathcal{F}$ that are of size less than $n-1$. Moreover, we assume that the elements that are missing in these two sets are disjoint. These assumptions specify a specific form of filter $\mathcal{F}=\{F_0,F_1,F_2,\ldots,F_n,F_{n+1},F_{n+2}\}$, where $F_0=[n]$, $F_1=[n]\setminus\{1\}$, $F_2=[n]\setminus\{2\}$, $\ldots$, $F_n=[n]\setminus\{n\}$, $F_{n+1}=[n]\setminus\{1,2\}$, $F_{n+2}=[n]\setminus\{3,4\}$. Let $\mathcal{S}=\{S_0,S_1,\ldots,S_{n+2}\}$, where $S_i$ is mapped to $F_i$ for $i=0,1,\ldots,n+2$. 
\begin{observation}
    \normalfont{(Note 8 \cite{raz}) }$S_{n+1}\neq\emptyset$ and $S_{n+2}\neq\emptyset$ 
\end{observation}
Before the next observation, we provide an equivalent formulation for the non-interference condition. This formulation is particularly important in the proof of Observation \ref{minsizeobs}.

\begin{lemma}
    Let $p$ and $q$ be arbitrary elements in ${0,1,\ldots,n+2}$. $[S_p, F_p]\cap[S_q, F_q]=\emptyset$ 
    if and only if there exists $i\in [n]$ such that at least one of the following is true:
    \begin{enumerate}
        \item $i\in S_p$ and $i\notin F_q$ 
        \item $i\in S_q$ and $i\notin F_p$
    \end{enumerate}
    In the case of both $p,q\in [n]$, this is the same as saying at least one of $q\in S_p$ and $p\in S_q$ is true.
\end{lemma}
\begin{proof}
First we show the forward direction using a contrapositive argument. For the sake of contradiction, assume that neither of the two statements are true. This implies $S_p \subseteq F_q$ and $S_q \subseteq F_p$, thus $S_p \cup S_q \in [S_p, F_p]\cap[S_q, F_q]$, so $[S_p, F_p]\cap[S_q, F_q]\neq\emptyset$. 

For the other direction, assume without loss of generality that there is some $i$ such that $i\in S_p$, and $i\notin F_q$. 
Then for all $P\in [S_p, F_p]$, we have $i\in P$. For all $Q\in [S_q, F_q]$, we have $i\notin Q$. Therefore $[S_p, F_p]\cap[S_q, F_q]=\emptyset$.
\end{proof}

\begin{observation}
\label{minsizeobs}
    Let $x$ be the smallest set size in $\mathcal{S}$, then $n\geq 4x+4$.
\end{observation}

\begin{proof}
    Let $p$ be an arbitrary element in $[n]$. Because $[S_p, F_p]\cap[S_q, F_q]=\emptyset$
    for all $q\in[n]\setminus\{p\}$ by the non-interference condition, the sum of $|S_p|$ and the number of appearances of element $p$ in all other $S_q$ must be at least $n-1$. Therefore, after accounting for double counting, we have $\sum_{p\in[n]}|S_p|\geq\frac{n(n-1)}{2}$. Moreover, because $[S_1, F_1]\cap[S_{n+1}, F_{n+1}]=\emptyset$, we have $2\in S_1$. Because $[S_2, F_2]\cap[S_{n+1}, F_{n+1}]=\emptyset$, we have $1\in S_2$. Because either one of $2\in S_1$ and $1\in S_2$ is necessary and sufficient for $[S_1, F_1]\cap[S_2, F_2]=\emptyset$, the simultaneous existence of $2\in S_1$ and $1\in S_2$ is not accounted in the previous argument of $\sum_{p\in[n]}|S_p|\geq\frac{n(n-1)}{2}$. Similarly, from $[S_3, F_3]\cap[S_{n+2}, F_{n+2}]=\emptyset$ and $[S_4, F_4]\cap[S_{n+2}, F_{n+2}]=\emptyset$, we have  $4\in S_3$ and $3\in S_4$. Therefore, we have $\sum_{p\in[n]}|S_p|\geq\frac{n(n-1)}{2}+2$. To fail the abundance condition, $\sum_{k\in\{0,1,\ldots,n+2\}}|S_k|\leq n(\frac{n}{2}+1)$ as $|\mathcal{S}|=n+3$ and $n$ is even. Because $|S_0|=n$ by observation \ref{unimapobs} 
    and $|S_{n+1}|\geq x$ and $|S_{n+2}|\geq x$, we have $\sum_{p\in[n]}|S_p|\leq n(\frac{n}{2}+1)-n-2x$. Solving for inequality $ n(\frac{n}{2}+1)-n-2x\geq\frac{n(n-1)}{2}+2$, we have $n\geq 4x+4$.
\end{proof}


\begin{observation}
    If there is a counterexample with such form of filter and satisfy $n=4x+4$, then each element in $[n]$ belongs to exactly $\frac{n}{2}+1$ sets in $\mathcal{S}$.
\end{observation}
\begin{proof}
    This is a direct result of observation \ref{minsizeobs}
    . In this case, the minimum and maximum value of $\sum_{p\in[n]}|S_p|$ is equal. Therefore, we are required to maximize $\sum_{p\in\{0,\ldots,n+2\}}|S_p|$ without satisfying the abundance condition.
\end{proof}

\section{Construction}\label{construction}
Counterexamples with such form of filter and satisfy $n=4x+4$ indeed exist. Raz offered a specific counterexample for $x=1$. This section offers a way to construct such counterexamples for any $x\geq 2$. The non-interference condition and the specific forms of $F_{n+1}$ and $F_{n+2}$ provide certain restrictions on $S_1$, $S_2$, $S_3$, and $S_4$. After constructing these sets, We then filled in the other sets with the help of the observations above. The final construction is given below. 
\\ \noindent $S_0=[n]$
\\ \noindent $S_1=\{2,\frac{n}{2}+1,\ldots,n\}$
\\ \noindent $S_2=\{1,3,\ldots,\frac{n}{2},n\}$
\\ \noindent $S_3=\{1,4,\ldots,\frac{n}{2}+2\}$
\\ \noindent $S_4=\{1,3,\frac{n}{2}+3,\ldots,n\}$
\\ \noindent $S_{k_1}=\{1,4,k_1+1,\ldots,\frac{n}{2}+k_1-2\}, 5\leq k_1 \leq \frac{n}{4}+2$
\\ \noindent $S_{k_2}=\{1,4,k_2+1,\ldots,\frac{n}{2}+k_2-3\}, \frac{n}{4}+3\leq k_2 \leq\frac{n}{2}$ 
\\ \noindent $S_{\frac{n}{2}+2}=\{2,4,\frac{n}{2}+3,\ldots,n-1\}$
\\ \noindent $S_{\frac{n}{2}+3}=\{2,3,\frac{n}{2}+4,\ldots,n\}$
\\ \noindent $S_{k_3}=\{2,3,5,\ldots,k_3-\frac{n}{2}+1,k_3+1,\ldots,n\},\frac{n}{2}+4\leq k_3 \leq\frac{3n}{4}$
\\ \noindent $S_{\frac{3n}{4}+1}=\{2,3,5,\ldots,\frac{n}{4}+3,\frac{3n}{4}+2,\ldots,n-1\}$
\\ \noindent $S_{k_4}=\{2,3,5,\ldots,k_4-\frac{n}{2}+2,k_4+1,\ldots,n\}, \frac{3n}{4}+2\leq k_4 \leq n-1$
\\ \noindent $S_n=\{3,5,\ldots,\frac{n}{2}+2,\frac{3n}{4}+1\}$
\\ \noindent $S_{n+1}=\{\frac{3n}{4}+2,\ldots,n\}$
\\ \noindent $S_{n+2}=\{1,5,\ldots,\frac{n}{4}+2\}$
\\ \noindent It can be checked by hand that the following hold. 
\begin{enumerate}
    \item The minimum set size of $\mathcal{S}$ in this construction is $\frac{n}{4}-1$, which is achieved by $S_{n+1}$ and $S_{n+2}$ 
    \item Reimer's conditions are satisfied.
    \item Each element in $[n]$ appears in exactly $\frac{n}{2}+1$ member sets of $\mathcal{S}$.
\end{enumerate}
The verification of non-interference condition is included in the appendix.

\section{Properties}
The family $\mathcal{S}$ created under the above construction is not union-closed. In other words, $\mathcal{S}\neq cl(\mathcal{S})$. When $x=2$, $\frac{|S|}{|cl(\mathcal{S})|}=\frac{15}{133}\approx0.1128$. When $x=3$, $\frac{|S|}{|cl(\mathcal{S})|}=\frac{19}{233}\approx0.0815$. When $x=4$, $\frac{|S|}{|cl(\mathcal{S})|}=\frac{23}{354}\approx0.0650$. Unless explicitly specified, the discussion in this section only applies to the type of counterexamples constructed in the previous section.
\begin{theorem}
If $\mathcal{S}$ is a family satisfying the construction given in Section \ref{construction}, then $|cl(\mathcal{S})|=\Theta(n^2)$. 
\end{theorem}
\begin{proof}
First we show $|cl(\mathcal{S})|=\mathcal{O}(n^2)$. Define four families $\mathcal{S}_{k_1}$, $\mathcal{S}_{k_2}$, $\mathcal{S}_{k_3}$, $\mathcal{S}_{k_4}$ as
\begin{itemize}
    \item[$\mathcal{S}_{k_1}$]=\{$S_{k_1}$:$5\leq k_1 \leq \frac{n}{4}+2$\}
    \item[$\mathcal{S}_{k_2}$]=\{$S_{k_2}$:$\frac{n}{4}+3\leq k_2 \leq\frac{n}{2}$\}
    \item[$\mathcal{S}_{k_3}$]=\{$S_{k_3}$:$\frac{n}{2}+4\leq k_3 \leq\frac{3n}{4}$\}
    \item[$\mathcal{S}_{k_4}$]=\{$S_{k_4}$:$\frac{3n}{4}+2\leq k_4 \leq n-1$\}
\end{itemize}
Note that there are only 11 other sets in $\mathcal{S}$. Therefore, proving 
$|cl(\mathcal{S}_{k_1}\cup\mathcal{S}_{k_2}\cup\mathcal{S}_{k_3}\cup\mathcal{S}_{k_4})|=\mathcal{O}(n^2)$ is sufficient to show $|cl(\mathcal{S})|=\mathcal{O}(n^2)$. To begin with, we have
\\$cl(\mathcal{S}_{k_{1}})=\{\{1,4,k_{1a},\ldots,k_{1b}\}: 6\leq k_{1a}\leq \frac{n}{4}+3, \frac{n}{2}+3\leq k_{1b}\leq \frac{3n}{4},k_{1b}-k_{1a}\geq\frac{n}{2}-3\}$
\\ \noindent$cl(\mathcal{S}_{k_{2}})=\{\{1,4,k_{2a},\ldots,k_{2b}\}: \frac{n}{4}+4\leq k_{2a}\leq \frac{n}{2}+1, \frac{3n}{4}\leq k_{2b}\leq n-3,k_{2b}-k_{2a}\geq\frac{n}{2}-4\}$
\\ \noindent$cl(\mathcal{S}_{k_{3}})=\{\{2,3,5,\ldots,k_{3a},k_{3b},\ldots,n\}: 5\leq k_{3a}\leq \frac{n}{4}+1, \frac{n}{2}+5\leq k_{3b}\leq \frac{3n}{4}+1,k_{3b}-k_{3a}\leq\frac{n}{2}\}$
\\ \noindent$cl(\mathcal{S}_{k_{4}})=\{\{2,3,5,\ldots,k_{4a},k_{4b},\ldots,n\}: \frac{n}{4}+4\leq k_{4a}\leq \frac{n}{2}+1, \frac{3n}{4}+3\leq k_{4b}\leq n ,k_{4b}-k_{4a}\leq\frac{n}{2}-1\}$
\\ \noindent Because $cl(\mathcal{S}_{k_{1}}\cup\mathcal{S}_{k_{2}})=cl(cl(\mathcal{S}_{k_{1}})\cup cl(\mathcal{S}_{k_{2}}))$, we have \\ \noindent $cl(\mathcal{S}_{k_{1}}\cup\mathcal{S}_{k_{2}})\subset\{\{1,4,k_{5a},\ldots,k_{5b}\}: 6\leq k_{5a}\leq \frac{n}{2}+1, \frac{n}{2}+3\leq k_{5b}\leq n-3,k_{5b}-k_{5a}\geq\frac{n}{2}-4\}$
Similarly, we have \\ \noindent $cl(\mathcal{S}_{k_{3}}\cup\mathcal{S}_{k_{4}})\subset\{\{2,3,5,\ldots,k_{6a},k_{6b},\ldots,n\}: 5\leq k_{6a}\leq \frac{n}{2}+1, \frac{n}{2}+5\leq k_{6b}\leq n ,k_{6b}-k_{6a}\leq\frac{n}{2}\}$
Then
\\ \noindent $cl(\mathcal{S}_{k_1}\cup\mathcal{S}_{k_2}\cup\mathcal{S}_{k_3}\cup\mathcal{S}_{k_4})\subset\mathcal{C}=\{\{1,\ldots,k_{8a},k_{7a},\ldots,k_{7b},k_{8b},\ldots,n\}:5\leq k_{8a}< k_{7a}\leq \frac{n}{2}+1,\frac{n}{2}+3\leq k_{7b}<k_{8b}\leq n,k_{7b}-k_{7a}\geq\frac{n}{2}-4,k_{8b}-k_{8a}\leq\frac{n}{2}\}$
\\ \noindent From the inequalities constraints, we have $k_{7b}\geq k_{7a}+\frac{n}{2}-4$ and $k_{8b}\leq k_{8a}+\frac{n}{2}< k_{7a}+\frac{n}{2}$. Therefore, the sets in $\mathcal{C}$ can be partitioned to two different types. The first type contains the sets where there is no gap between $k_{7b}$ and $k_{8b}$; that is, they are of the form $\{1,\ldots,k_{8a},k_{7a},\ldots,n\}$. The number of these sets is $\mathcal{O}(n^2)$, as each of $k_{8a}$ and $k_{7a}$ have linear in $n$ choices. 
For the sets where the gap between $k_{7b}$ and $k_{8b}$ does exist, for any fixed $k_{7a}$, there are at most 6 choices for the pair $(k_{7b},k_{8b})$ due to the two inequalities above. Therefore, the number of sets that belong to the second type is also $\mathcal{O}(n^2)$. Thus $|\mathcal{C}|=\mathcal{O}(n^2)$.
From the description of $cl(\mathcal{S}_{k_1})$ above, we can obtain that $cl(\mathcal{S}_{k_1}) = \Omega(n^2)$ and thus $cl(\mathcal{S}) = \Theta(n^2)$. 
\end{proof}

\begin{corollary}
   $\lim_{x\to\infty}\frac{|\mathcal{S}|}{|cl(\mathcal{S})|}=0$.
\end{corollary}
\begin{proof}
    This follows directly from $|\mathcal{S}|=\Theta(n)$ and $|cl(\mathcal{S})|=\Theta(n^2)$.
\end{proof}

\section{Conjectures}
\begin{conjecture}
For $x\geq5$, $|cl(\mathcal{S})|=\frac{23}{32}n^2+\frac{35}{8}n-21$.
\end{conjecture}
\noindent This conjecture has been verified for $x=5,6,\ldots,12$. Links to the scripts used in this verificaiton can be found in the appendix. 

\newpage
\section{Appendix}
There are two codes linked below. \href{https://1drv.ms/u/c/ecb46a0d19fbe4ee/EZX3hGbeXaRJvJXeveP9JYEBf3kLJcPtKVBGuBy6Aukpaw?e=Bk9aZD}{generator.cpp} generates families and filters with the construction given in this paper. The program asks for $x$, the minimum set size, as input. The program outputs a file where each row corresponds to a set. The set elements are represented in binary. For example, when $n=8$, $\{1,3,4,7,8\}$ will be represented as ``1 0 1 1 0 0 1 1" and $\{2,4\}$ will be represented as ``0 1 0 1 0 0 0 0". When a row ends with ``S", the set is in the family. The set immediately after it, ending with ``F", is its correspond set in the filter. The sets are in the same order as seen in the construction section previously.
\\ \href{https://1drv.ms/u/c/ecb46a0d19fbe4ee/EYzgZrJwlCdBppdLFCrLY-sBvYlWfbIerNS36Ni5ZSUoFQ?e=LF1ces}{uclosure.cpp} accepts an input file in the same format as the output file from ``generator.cpp"; that is, families and filters are represented in binary, and each set in the family is followed immediately by its corresponding set in the filter. The program then checks if the non-interference condition is violated, and if there is any element satisfying the abundance condition. The program also generates the closure of the family. It counts the size of the closure and the numbers of sets in the closure of each specific size.

\end{document}